\documentclass{amsart}%
\usepackage{amsfonts}
\usepackage{graphicx}
\usepackage{graphics,color}
\usepackage{amsmath}
\usepackage{amssymb}
\usepackage{morefloats}%
\providecommand{\U}[1]{\protect\rule{.1in}{.1in}}
\newtheorem{theorem}{Theorem}
\theoremstyle{plain}

\newtheorem{definition}{Definition}

\newtheorem{lemma}{Lemma}

\newtheorem{proposition}{Proposition}
\newtheorem{remark}{Remark}

\numberwithin{equation}{section}
\begin{document}
\title[Ensemble Eddy Viscosity Energy Dissipation]{On the energy dissipation rate of ensemble eddy viscosity models of
turbulence: Shear flows }
\author{William Layton }
\address{Dept. of Mathematics, Univ. of Pittsburgh, Pittsburgh, PA 15260, USA}
\email{wjl@pitt.edu}
\urladdr{}
\thanks{The research herein was partially supported by NSF grant DMS 2410893}
\thanks{}
\subjclass[2000]{Primary 76F02; Secondary 35Q30}
\keywords{turbulence, Navier Stokes, eddy viscosity, ensemble}
\dedicatory{ }
\begin{abstract}
\ Classical eddy viscosity models add a viscosity term with turbulent
viscosity coefficient developed beginning with the Kolmogorov-Prandtl
parameterization. Approximations of unknown accuracy of the unknown mixing
lengths and turbulent kinetic energy are typically constructed by solving
associated systems of nonlinear convection-diffusion-reaction equations with
nonlinear boundary conditions. Often these over-diffuse so additional fixes
are added such as wall laws or using different approximations in different
regions (which must also be specified). Alternately, one can solve an ensemble
of NSE's with perturbed data, compute the ensemble mean and fluctuation and
simply compute directly the turbulent viscosity parameterization. This idea is
recent, seems to be of lower complexity and greater accuracy and produces
parameterizations with the correct near wall asymptotic behavior. The question
then arises: Does this ensemble eddy viscosity approach over-diffuse
solutions? This question is addressed herein.

\end{abstract}
\maketitle

\section{Introduction}

One failure mode of eddy viscosity models of turbulence is over dissipation,
leading to a lower Reynolds number (even laminar) solution. Over-dissipation
can be caused by too large turbulent viscosity $\mathbf{\nu}_{turb}$ in a flow
interior and by too large $\mathbf{\nu}_{turb}$\ in the near wall boundary
layer region where $\nabla u$\ is large. The first possibility was analyzed in
a previous paper \cite{LR26}. Here we extend the analysis of energy
dissipation of ensemble eddy viscosity models to shear flows, addressing the
near wall region.

Turbulent flow simulations with uncertain data require computing velocity and
pressure ensembles \cite{K03}, $u_{j}=u(x,t;\omega_{j}),p_{j}=p(x,t;\omega
_{j}),j=1,\cdot\cdot\cdot,J$ , of (typically) an eddy
viscosity\footnote{Notation: To reduce non-essential notaion we have replaced
the mechanically correct deformation tensor in the eddy viscosity term with
the full gradient. Also, with slight abuse of notation,$x$ will denote
$(x,y,z)$ and its first component.} model
\begin{gather}
u_{t}+u\cdot\nabla u-\nu\triangle u-\nabla\cdot\left(  \nu_{turb}(\cdot)\nabla
u\right)  +\nabla p=0\text{, in }\Omega\text{, }%
j=1,...,J,\label{eq:EnsembleNSE}\\
\nabla\cdot u=0\text{, and }u(x,0;\omega_{j})=u_{j}^{0}(x)\text{, in }%
\Omega\text{ and }u=0\text{, on }\partial\Omega\text{.}\nonumber
\end{gather}
Here $\nu$ is the kinematic viscosity and $\omega_{j}$\ are the sampled values
determining the ensemble data. The domain is $\Omega=(0,L)^{3}$ with
$L-$periodic boundary conditions in $x,y$. \ The ensemble mean\textbf{\ }%
$\left\langle u\right\rangle _{e}$\textbf{,} fluctuation\textbf{\ }$u^{\prime
}$, its magnitude $|u^{\prime}|_{e}$, and induced turbulent kinetic energy
(TKE) density $k^{\prime}$\ are%
\begin{gather*}
\text{ensemble average: }\left\langle u\right\rangle _{e}:=\frac{1}{J}%
\sum_{j=1}^{J}u(x,t;\omega_{j}),\text{ fluctuation: }u^{\prime}(x,t;\omega
_{j}):=u-\left\langle u\right\rangle _{e}\text{,}\\
|u^{\prime}|_{e}^{2}:=\frac{1}{J}\sum_{j=1}^{J}|u^{\prime}|^{2}\text{ and TKE
density: \ }k^{\prime}(x,t):=\left\langle \frac{1}{2}\text{ }|u^{\prime}%
|^{2}\right\rangle _{e}=\frac{1}{2}|u^{\prime}|_{e}^{2}.
\end{gather*}
The eddy viscosity is given by the Kolmogorov-Prandtl relation in terms of the
turbulence length scale $l(x,t)$\ and the turbulent kinetic energy $k^{\prime
}(x,t)$:
\[
\nu_{turb}\mathbf{(\cdot)}\text{ }\mathbf{=}\text{ }\mu l\sqrt{k^{\prime}%
},\text{ for }\mu\text{ \ an }\mathcal{O}(1)\text{ parameter.}%
\]
Here $\tau$ is a selected turbulence time scale chosen smaller than the large
scale turnover time $T^{\ast},$ $\tau\leq T^{\ast}$,

The shear boundary conditions at $z=0,L$ are no slip for, respectively, a
fixed and moving wall:%
\begin{equation}
u(x,y,0;\omega_{j},t)=(0,0,0)\text{ and }u(x,y,L;\omega_{j},t)=(U,0,0).
\label{eq:ShearBC}%
\end{equation}
The global length and velocity scales are $L,U$ and the usual Reynolds number
is $\mathcal{R}e=LU/\nu$. \ From the energy inequality (below), the energy
dissipation rate $\varepsilon_{\text{model}}$\ is%
\[
\varepsilon_{\text{model}}:=\left\langle \frac{1}{|\Omega|}\int_{\Omega}%
\nu|\nabla u(x,t;\omega_{j})|^{2}+\nu_{turb}(\cdot)|\nabla u(x,t;\omega
_{j})|^{2}dx\right\rangle _{e}.
\]

With an ensemble of solutions, $k^{\prime}$ can simply be calculated as above
without approximation or modelling, \cite{JL16}, \cite{JL14}. The length scale
$l=\sqrt{2k^{\prime}}\tau$ (and in \cite{LM18}, \cite{SAA92}, \cite{K42}) is
determined from the use of a turbulence model for numerical simulations with
fixed time or space resolution. Thus, we assume that in an under-resolved
simulation we are given a small time scale (related to the time step) denoted
$\tau$. We select the length scale $l(x,t)$ to be the distance a fluctuation
travels in time $\tau$%
\[
l=\sqrt{2k^{\prime}(x,t)}\tau\text{ so that }\nu_{turb}=\mu|u^{\prime}%
|_{e}^{2}\tau
\]
This choice was mentioned by Prandtl in 1926 and independently used by
Kolmogorov in his URANS model \cite{K42}. The shear boundary conditions mean
that the near wall \ region, where velocity gradients are large, will have a
different character than the interior of the flow. There are different ways
this different character can be reflected in the model and in the analysis.
For example, the final result (below) will take a different value of the
parameter $\mu$\ in the near wall region and in the flow interior.

The primary failure model of eddy viscosity is over dissipation of solutions,
e.g., \cite{Sagaut}. There are significant phenomenological reasons, Section
1.1, to hope that with $\nu_{turb}=\mu|u^{\prime}|_{e}^{2}\tau$
(\ref{eq:EnsembleNSE}) does not over dissipate model solutions, i.e., that its
energy dissipation rate is comparable uniformly in the Reynolds number to the
$\mathcal{O}(U^{3}/L)$ energy input rate. To study dissipation due to near
wall effects we follow the plan in Doering and Constantin \cite{CD92} and
study shear flows in the simplest geometry. Extension of the NSE results to
general shear flows in general domains was accomplished in Wang \cite{Wang97};
we conjecture that similar extensions of the results herein are possible.

\subsection{The energy inequality}

The analysis herein treats weak solutions to the model satisfying a standard
energy inequality. To explain, let $\phi(x,y,z)$\ denote a divergence free
function with $\phi,\nabla\phi\in L^{2}(\Omega)$ and satisfying the shear
boundary conditions (\ref{eq:ShearBC}). A standard energy \textit{equality} is
obtained (for sufficiently regular solutions) by integrating the dot product
of the model with $u-\phi$. Motivated by the resulting energy equality, we
assume that for any divergence free $u_{0},f\in L^{2}$ a weak solution of the
model (\ref{eq:EnsembleNSE}) with shear boundary conditions (\ref{eq:ShearBC})
exists\footnote{Existence theory for (\ref{eq:ShearBC}) is not yet
established. For periodic boundary conditions, it is developed in Section 4 of
\cite{LR26}.} for each realization and satisfies the associated energy
\textit{inequality}. Specifically, for any divergence free function
$\phi(x,y,z)$ with $\phi,\nabla\phi\in L^{2}(\Omega)$ and satisfying
(\ref{eq:ShearBC}), we assume%
\begin{gather}
\frac{1}{2}\frac{d}{dt}||v||^{2}+\int_{\Omega}2[\nu+\,\nu_{turb}]|\,\nabla
^{s}{v}|^{2}dx\leq\label{eq:EnIneq}\\
(v_{t},\phi)+\int_{\Omega}2[\nu+\,\nu_{turb}]\nabla^{s}{v}:\nabla^{s}\phi
dx+(v\cdot\nabla v,\phi).\nonumber
\end{gather}
This implies, in particular, the ensemble averaged energy inequality%
\begin{gather}
\left\langle \frac{1}{2}\frac{1}{|\Omega|}\int_{\Omega}|u(x,T;\omega_{j}%
)|^{2}dx\right\rangle _{e}+\nonumber\\
+\int_{0}^{T}\left\langle \frac{1}{|\Omega|}\int_{\Omega}\nu|\nabla
u(x,t;\omega_{j})|^{2}+\mu\tau|u^{\prime}(x,t)|_{e}^{2}|\nabla u(x,t;\omega
_{j})|^{2}dx\right\rangle _{e}dt=\\
\frac{1}{2}\left\langle \frac{1}{|\Omega|}\int_{\Omega}|u_{0}(x;\omega
_{j})|^{2}dx\right\rangle _{e}+\int_{0}^{T}\left\langle \frac{1}{|\Omega|}%
\int_{\Omega}f(x,t;\omega_{j})\cdot u(x,t;\omega_{j})dx\,\right\rangle
_{e}dt\text{ }.\nonumber
\end{gather}
Since $|u^{\prime}|_{e}^{2}$\ is independent of $\omega_{j}$,\ the turbulent
viscosity $\,\nu_{turb}$\ will be the same for all realizations (i.e.,
independent of $\omega_{j}$). This property was referred to as
universality\ in Carati, Roberts and Wray \cite{CRW02}. It was fundamental to
reducing the computational cost of solving for an ensemble of model solutions
in the algorithms developed in \cite{JL14}.

We present several results (summarized next and proven in Section 3) and a
collection of open problems (in Section 4). The effective viscosity,
$\nu_{eff}$, and effective Reynolds number are defined in a standard manner,
Definition 2 in Section 2. The (upper) near wall region $\mathcal{S}_{\beta}$
is denoted
\[
\mathcal{S}_{\beta}=\left\{  (x,y,z):0\leq x\leq L,0\leq y\leq L,(1-\beta
)L<z<L\right\}  ,\beta=\frac{1}{8}\mathcal{R}e_{eff}^{-1}.
\]
The first theorem in Section 3, is that energy dissipation is governed by the
ratio of viscosities in the near wall region.

\begin{theorem}
Any weak solution of the eddy viscosity model (\ref{eq:EnsembleNSE})
satisfying the energy inequality (\ref{eq:EnIneq}) has its model energy
dissipation bounded as%
\begin{gather}
\lim\sup_{T\rightarrow\infty}\frac{1}{T}\int_{0}^{T}\varepsilon_{\text{model}%
}dt\leq\\
\leq\left(  \frac{5}{2}+16\frac{\nu}{\nu_{eff}}+32\lim\sup_{T\rightarrow
\infty}\frac{1}{T}\int_{0}^{T}\frac{1}{|\mathcal{S}_{\beta}|}\int
_{\mathcal{S}_{\beta}}\,\frac{\nu_{turb}}{\nu_{eff}}dxdt\right)  \frac{U^{3}%
}{L}.\nonumber
\end{gather}

\end{theorem}

We expect $u^{\prime}$\ and thus $\nu_{turb}$\ to be small in $\mathcal{S}%
_{\beta}$ (being constrained by the wall) so this estimate is hopeful. The
second theorem in Section 3 reflects the anisotropic nature of the near wall
region. It also suggests an anisotropic turbulent viscosity may be appropriate
in the near wall region with smaller coefficient in the wall normal direction.

\begin{theorem}
Let $T^{\ast}$\ denote the large scale \ turnover time. Any weak solution of
the eddy viscosity model (\ref{eq:EnsembleNSE}) satisfying the energy
inequality (\ref{eq:ShearEnergyIneq}) has its model energy dissipation bounded
as%
\begin{gather}
\lim\sup_{T\rightarrow\infty}\frac{1}{T}\int_{0}^{T}\varepsilon_{\text{model}%
}dt\leq\left(  \frac{5}{2}+16\frac{\nu}{\nu_{eff}}\right)  \frac{U^{3}}%
{L}+\nonumber\\
+\left(  \frac{\sqrt[3]{3}}{6}C_{26}^{2}\right)  \mu\frac{\tau}{T^{\ast}}%
\lim\sup_{T\rightarrow\infty}\frac{1}{T}\int_{0}^{T}\left(  \frac
{1}{|\mathcal{S}_{\beta}|}\int_{\mathcal{S}_{\beta}}\nu_{eff}\left\langle
\,\left\vert \frac{\partial u^{\prime}}{\partial z}\right\vert ^{2}%
\right\rangle _{e}dx\right)  dt
\end{gather}

\end{theorem}

To obtain an estimate independent of model solution, the integral over
$\mathcal{S}_{\beta}$\ on the RHS must be replaced by one over $\Omega$,
$\frac{\partial u}{\partial z}$\ replaced by $\nabla u$\ , and the resulting
term subsumed in the LHS. The (pessimistic) result (concluding Section 3) is
the following.

\begin{theorem}
Suppose $\mu$ is one constant value in the flow interior and $\mu=\mu_{\beta}$
in $\mathcal{S}_{\beta}$. If
\[
\mu_{\beta}\leq0.270\,64\mathcal{R}e^{-1}%
\]
then
\[
\left\langle \varepsilon_{\text{model}}\right\rangle _{\infty}\leq\left(
5+32\frac{\nu}{\nu_{eff}}\right)  \frac{U^{3}}{L}.
\]

\end{theorem}

\subsection{Related work}

We cannot over-stress the importance of the work of Constantin, Doering and
Foias \cite{CD92}, \cite{DF02} on the NSE to the analysis herein. Their work
(building on \cite{B78}, \cite{H72}, \cite{H55}) has been developed \ in many
important directions for the NSE in subsequent years. For turbulence models,
upper bounds for energy dissipation rates, inspired by \cite{CD92},
\cite{DF02}, address the most important failure mode of model over
dissipation, in,e.g., \cite{P17}, \cite{KLS21}, \cite{KLS22}, \cite{LS20}.

Many studies have computed flow ensembles of various turbulence models for
various applications (e.g., \cite{CRW02}, \cite{LM18}, \cite{MT01},
\cite{MX06}). Herein eddy viscosity models using ensemble data for
parameterization, as proposed in Carati, Roberts and Wray \cite{CRW02} and
developed in \cite{J14}, \cite{JL16}, \cite{JL14}, are analyzed. In these, the
turbulent viscosity term replaces the Reynolds stresses that have the near
wall behavior $\mathcal{O}(d^{2})$ as the wall-normal distance $d\rightarrow
0.$ Since gradients are large in the boundary layer region, the near wall
$\mathcal{O}(d^{2})$ should be replicated in $\nu_{turb}$ to reduce the chance
of over dissipation of solutions, Pakzad \cite{P17}. For the EEV model, near
wall asymptotics of $\nu_{turb}$ depend on $|u^{\prime}(x,t)|_{e}^{2}$. Since
$u=0$ at the wall, a formal Taylor series expansion suggests the model studied
herein may satisfy this requirement.

\section{Notation and preliminaries}

In the analysis herein, the flow domain will be the open box $\Omega
=(0,(0,L_{\Omega})^{3}$ in $\mathbb{R}^{3}$. The $L^{2}(\Omega)$ norm and the
inner product are $\Vert\cdot\Vert$ and $(\cdot,\cdot)$. Likewise, the
$L^{p}(\Omega)$ norms and the Sobolev $W_{p}^{k}(\Omega)$ norms are
$\Vert\cdot\Vert_{L^{p}}$ and $\Vert\cdot\Vert_{W_{p}^{k}}$ respectively.
$H^{k}(\Omega)$ is the Sobolev space $W_{2}^{k}(\Omega)$, with norm
$\Vert\cdot\Vert_{k}$. $C$ represents a generic positive constant independent
of $\nu,U,L,$ and pother model parameters. Its value may vary from situation
to situation. For $v=v(x,t;\omega_{j}),$ recall that $|v|_{e}^{2}%
(x,t):=\frac{1}{J}\sum_{j=1}^{J}|v|^{2}.$

Three kinds of averaging, ensemble, finite time and infinite time, will be
used. Ensemble averaging, already introduced, is $\left\langle \phi
\right\rangle _{e}:=\frac{1}{J}\sum_{j=1}^{J}\phi(\omega_{j})$. The short and
long time average of a function $\phi(t)$\ are respectively%
\[
\left\langle \phi\right\rangle _{T}=\frac{1}{T}\int_{0}^{T}\phi(t)dt.\text{
and }\left\langle \phi\right\rangle _{\infty}=\lim\sup_{T\rightarrow\infty
}\frac{1}{T}\int_{0}^{T}\phi(t)dt.
\]
The next lemma follows by inserting and rearranging the averages.

\begin{lemma}
All three averages, $\left\langle \phi\right\rangle _{e},\left\langle
\phi\right\rangle _{T}$\ and \ $\left\langle \phi\right\rangle _{\infty}$
satisfy
\[
\left\langle \phi\psi\right\rangle \leq\left\langle |\phi|^{2}\right\rangle
^{1/2}\left\langle |\psi|^{2}\right\rangle ^{1/2}\text{ , }\left\langle
\left\langle \phi\right\rangle _{e}\right\rangle _{T}=\left\langle
\left\langle \phi\right\rangle _{T}\right\rangle _{e}\text{ , }\left\langle
\left\langle \phi\right\rangle _{e}\right\rangle _{\infty}=\left\langle
\left\langle \phi\right\rangle _{\infty}\right\rangle _{e}.
\]

\end{lemma}

To develop the results some scaling constants will be needed.

\begin{definition}
The large velocity scale for shear flow is clearly the lid's shear velocity
(denoted $U$). The fluctuation scale $U^{\prime}$, large scale turnover time
$T^{\ast}$,\ and Reynolds number $\mathcal{R}e$ are
\begin{equation}
\text{\ }U^{\prime}=\left\langle \left\langle \frac{1}{|\Omega|}||u^{\prime
}||^{2}\right\rangle _{e}\right\rangle _{\infty}^{\frac{1}{2}},T^{\ast}%
=\frac{L}{U}\text{ and }\mathcal{R}e=\frac{LU}{\nu} \label{eq:ULscales}%
\end{equation}

\end{definition}

Under the assumption of the energy inequality, the above quantities are well
defined and finite.

\textbf{The energy dissipation rate}. The model's ensemble averaged energy
dissipation rate from the energy inequality is%
\[
\varepsilon(t):=\left\langle |\Omega|^{-1}\int_{\Omega}\nu|\nabla
u(x,t;\omega_{j})|^{2}+\mu\tau|u^{\prime}(x,t)|_{e}^{2}|\nabla u(x,t;\omega
_{j})|^{2}dx\right\rangle _{e}%
\]
It will be convenient to decompose the energy dissipation rate by%
\begin{gather*}
\varepsilon(t)=\varepsilon_{viscous}(t)+\varepsilon_{turb}(t)\text{ where}\\
\varepsilon_{viscous}(t)=\left\langle |\Omega|^{-1}\int_{\Omega}\nu|\nabla
u(x,t;\omega_{j})|^{2}dx\right\rangle _{e}\text{,}\\
\text{ }\varepsilon_{turb}(t)=\left\langle |\Omega|^{-1}\int_{\Omega}\mu
\tau|u^{\prime}(x,t;\omega_{j})|_{e}^{2}|\nabla u(x,t;\omega_{j}%
)|^{2}dx\right\rangle _{e}.
\end{gather*}

\textbf{The Hardy inequality.} With shear boundary conditions imposed at the
top ($z=L$) and bottom ($z=0$), the key idea in the analysis of shear flows is
to use $L^{p}-L^{q}$ generalizations of the Hardy inequality (responding to
the dominant nonlinearity in the problem) to connect the near-wall solution
behavior to the eddy viscosity coefficient. The basic, 1925, Hardy inequality
is that for any $F(x)$ with all integrals finite and $F(0)=0$, and any
$1<p<\infty$,%
\begin{equation}
\int_{0}^{\infty}\left\vert \frac{F(x)}{x}\right\vert ^{p}dx\leq\left(
\frac{p}{p-1}\right)  ^{p}\int_{0}^{\infty}\left\vert F_{x}(x)\right\vert
^{p}dx. \label{HardyIneq}%
\end{equation}
There have been many extensions and generalizations of the Hardy inequality.
In the analysis of Section 3 we will use the following $L^{p}-L^{q}$
extension, example 1.1 equation (1.6) in Person and Samko \cite{PS24}, also
\cite{PKS17}. For the parameters used, the optimal constant $C_{pq}$,
(\ref{eq:OptimalCpq})\ below, was derived in Bliss \cite{B30}. They establish
that, for all (non-negative) measurable functions, $1<p\leq q<\infty$ and for
parameters satisfying $\frac{\alpha+1}{q}=\frac{1}{p}-1$
\begin{equation}
\left(  \int_{0}^{\infty}x^{\alpha}\left(  \int_{0}^{x}f(t)dt\right)
^{q}dx\right)  ^{1/q}\leq C_{pq}\left(  \int_{0}^{\infty}f^{p}(t)dx\right)
^{1/p}. \label{eq:HardyLpLq}%
\end{equation}

\section{Shear Flows}

Over dissipation is caused by incorrect values of $\nu_{turb}$ in regions of
small scales, i.e. where $\nabla v$ is large. These small scales are generated
in the boundary layer and in the interior by breakdown of large scales through
the nonlinearity. This section considers those generated predominantly in the
turbulent boundary layer, studied via shear boundary conditions. Shear flows
can develop several ways. Inflow boundary conditions can emulate a jet of
water entering a vessel. A body force $f(\cdot)$ can be specified to be
non-zero large and tangential at a fixed wall. The simplest (chosen herein and
inspired by \cite{CD92}, \cite{Wang97}) is a moving wall modelled by a
boundary condition $v=g$ on the boundary where $g\cdot\widehat{n}=0$. This
setting includes flows between rotating cylinders. We impose $L-$periodic
boundary conditions in $x,y$, a fixed-wall no-slip condition at $z=0$ and a
wall at $z=L$ moving with velocity $(U,0,0)$:%
\begin{equation}%
\begin{array}
[c]{cc}%
Boundary & Conditions:\\
\text{moving top lid:} & u(x,y,L,t;\omega_{j})=(U,0,0)\\
\text{fixed bottom wall:} & u(x,y,0,t;\omega_{j})=(0,0,0)\\
\text{periodic side walls:} &
\begin{array}
[c]{c}%
u(x+L,y,z,t;\omega_{j})=u(x,y,z,t;\omega_{j}),\\
u(x,y+L,z,t;\omega_{j})=u(x,y,z,t;\omega_{j}).
\end{array}
\end{array}
\label{eq:Shear}%
\end{equation}
Herein, we assume that a weak solution of the model (\ref{eq:EnsembleNSE})
with shear boundary conditions (\ref{eq:Shear}) exists for each realization
and satisfies the usual energy inequality. Specifically, for any divergence
free function $\phi(x,y,z)$ with $\phi,\nabla\phi\in L^{2}(\Omega)$ and
satisfying the shear boundary conditions (\ref{eq:Shear}),%
\begin{gather}
\frac{1}{2}\frac{d}{dt}||u||^{2}+\int_{\Omega}[\nu+\,\nu_{turb}]|\,\nabla
u|^{2}dx\leq\label{eq:ShearEnergyIneq}\\
(u_{t},\phi)+\int_{\Omega}[\nu+\,\nu_{turb}]\nabla u:\nabla\phi dx+(u\cdot
\nabla u,\phi).\nonumber
\end{gather}
To prepare the proof of the main result, we recall from e.g. equations (20),
(22) in Doering and Constantin \cite{CD92}, that uniform bounds follow from
(\ref{eq:ShearEnergyIneq}) for both the NSE and (\ref{eq:EnsembleNSE}).

\begin{proposition}
[Uniform Bounds]Consider the model (\ref{eq:EnsembleNSE}) with shear boundary
conditions (\ref{eq:Shear}). For a weak solution satisfying
(\ref{eq:ShearEnergyIneq})\ the following are finite and bounded uniformly in
$T$
\begin{gather*}
||u(T)||^{2},\text{ \ }\int_{\Omega}\,\nu_{turb}(\cdot,T)dx,\text{ \ }\frac
{1}{T}\int_{0}^{T}\left(  \int_{\Omega}|\,\nabla u|^{2}dx\right)
dt\text{\ },\\
\text{ and }\frac{1}{T}\int_{0}^{T}\left(  \int_{\Omega}[\nu+\,\nu
_{turb}]|\,\nabla u|^{2}dx\right)  dt.
\end{gather*}

\end{proposition}

To formulate our first main result we recall the definition of the
\textbf{effective viscosity} $\nu_{eff}$ ($\geq\nu$) and a few related
quantities. The limit superiors in the infinite time averages $\left\langle
\cdot\right\rangle _{\infty}$\ are finite due to the uniform bounds above.

\begin{definition}
The \textbf{effective viscosity}\textit{ }$\nu_{eff}$ is
\[
\nu_{eff}:=\frac{\left\langle \left\langle \frac{1}{|\Omega|}\int_{\Omega}%
[\nu+\nu_{turb}]|\nabla u|^{2}dx\right\rangle _{e}\right\rangle _{\infty}%
}{\left\langle \left\langle \frac{1}{|\Omega|}\int_{\Omega}|\nabla
u|^{2}dx\right\rangle _{e}\right\rangle _{\infty}}.
\]
The large scale \textbf{turnover time} is $T^{\ast}=L/U$. The \textbf{Reynolds
number} and \textbf{effective Reynolds number} are $\mathcal{R}e=U\,L/\nu$
and\ $\mathcal{R}e_{eff}=U\,L/\nu_{eff}.$ Let $\beta=\frac{1}{8}%
\mathcal{R}e_{eff}^{-1}$ and denote the \textbf{near wall region}
$\mathcal{S}_{\beta}$ by
\[
\mathcal{S}_{\beta}=\left\{  (x,y,z):0\leq x\leq L,0\leq y\leq L,(1-\beta
)L<z<L\right\}  .
\]

\end{definition}

We can now present and prove the first result.

\begin{theorem}
Any weak solution of the eddy viscosity model (\ref{eq:EnsembleNSE})
satisfying the energy inequality (\ref{eq:ShearEnergyIneq}) has its model
energy dissipation bounded as%
\begin{equation}
\left\langle \varepsilon_{\text{model}}\right\rangle _{\infty}\leq\left(
\frac{5}{2}+16\frac{\nu}{\nu_{eff}}+32\left\langle \frac{1}{|\mathcal{S}%
_{\beta}|}\int_{\mathcal{S}_{\beta}}\,\frac{\nu_{turb}}{\nu_{eff}%
}dx\right\rangle _{\infty}\right)  \frac{U^{3}}{L}. \label{eq:ShearEstimate}%
\end{equation}

\end{theorem}

\begin{remark}
Before beginning the proof, we record a few observations. First, the result
shows the critical importance of the behavior of the turbulent viscosity in he
near wall region $\mathcal{S}_{\beta}$. Note that the average value over
$\Omega$\
\[
\left\langle \frac{1}{|\mathcal{\Omega}|}\int_{\mathcal{\Omega}}\,\frac
{\nu_{turb}}{\nu_{eff}}dx\right\rangle _{\infty}\leq1.
\]
If the average value of $\nu_{turb}/\nu_{eff}$ in $\mathcal{S}_{\beta}$\ (not
$\Omega$) is bounded uniformly in the Reynolds number then non-over
dissipation of the model follows.

Let $\widetilde{z}$ denote the distance to the top wall $z=L$. A formal Taylor
expansion shows that in $\mathcal{S}_{\beta}:$ $u^{\prime}(z)=\mathcal{O}%
(\widetilde{z})$. Thus, $\nu_{turb}=\mu\tau\,|u^{\prime}(x,t)|_{e}^{2}=\mu
\tau\mathcal{O}(\widetilde{z}^{2})$ there. This leads to the (incorrect)
heuristic prediction that%
\begin{align*}
\frac{1}{|\mathcal{S}_{\beta}|}\int_{\mathcal{S}_{\beta}}\,\frac{\nu_{turb}%
}{\nu_{eff}}dx  &  =\mu\tau\nu_{eff}^{-1}\frac{1}{|\mathcal{S}_{\beta}|}%
\int_{\mathcal{S}_{\beta}}\,\mathcal{O}(\widetilde{z}^{2})dxdydx=\mu\tau
\nu_{eff}^{-1}\frac{1}{|\mathcal{S}_{\beta}|}L^{3}\,\mathcal{O}(\beta^{3})\\
&  =\mu\tau\nu_{eff}^{-1}\,\mathcal{O}(\beta^{2})=\mathcal{O}(\mathcal{R}%
e_{eff}^{-1}).
\end{align*}
The reason this argument is incorrect is that the constant in "$\mathcal{O}$"
involves $u_{z}$ which (plausibly) grows like $\mathcal{O}(\mathcal{R}%
e_{eff})$ in the near wall region. Accounting for this, we obtain the
heuristic prediction $\nu^{-1}\frac{1}{|\mathcal{S}_{\beta}|}\int
_{\mathcal{S}_{\beta}}\,\,\frac{\nu_{turb}}{\nu_{eff}}dx=\mathcal{O}%
(\mathcal{R}e_{eff}^{+1})$.
\end{remark}

We now give the proof of the theorem.

\begin{proof}
Following Doering and Constantin \cite{CD92}, choose $\phi(z)=[\widetilde
{\phi}(z),0,0]^{T}$\ in the energy inequality where
\[
\widetilde{\phi}(z)=\left\{
\begin{array}
[c]{cc}%
0, & z\in\lbrack0,L-\beta\,L]\\
\frac{U}{\beta\,L}(z-(L-\beta\,L)), & z\in\lbrack L-\beta\,L,L]
\end{array}
\right.  \beta=\frac{1}{8}\mathcal{R}e_{eff}^{-1}.
\]
This function $\phi(z)$ is piecewise linear, continuous, divergence free and
satisfies the boundary conditions. The following are easily calculated values
\[%
\begin{array}
[c]{cc}%
||\,\phi\,||_{L^{\infty}(\Omega)}=U, & ||\,\nabla\phi\,||_{L^{\infty}(\Omega
)}=\frac{U}{\beta\,L},\text{ }\\
||\,\phi\,||^{2}=\frac{1}{3}\,U^{2}\,\beta\,L^{3}, & \text{ }||\,\nabla
\,\phi\,||^{2}=\frac{U^{2}\,L}{\beta}.
\end{array}
\]
With this choice of $\phi$, time averaging the energy inequality
(\ref{eq:ShearEnergyIneq}) over $[0,T]$ and normalizing by $|\Omega|=L^{3}$
gives
\begin{gather}
\frac{1}{2TL^{3}}||v(T)||^{2}+\frac{1}{T}\int_{0}^{T}\left(  \frac{1}{L^{3}%
}\int_{\Omega}[\nu+\nu_{turb}]|\,\nabla{v}|^{2}dx\right)  dt\leq\frac
{1}{2TL^{3}}||v(0)||^{2}\nonumber\\
+\frac{1}{TL^{3}}(v(T)-v(0),\phi)+\left\langle \frac{1}{L^{3}}(v\cdot\nabla
v,\phi)\right\rangle _{T}\\
+\frac{1}{T}\int_{0}^{T}\left(  \frac{1}{L^{3}}\int_{\Omega}[\nu+\,\nu
_{turb}]\nabla{v}:\nabla\phi dx\right)  dt.\nonumber
\end{gather}
Due to the above uniform in $T$ bounds, the time averaged energy inequality
can be expressed as%
\begin{align}
\frac{1}{T}\int_{0}^{T}\varepsilon_{\text{model}}dt &  \leq\mathcal{O}%
(\frac{1}{T})+\frac{1}{T}\int_{0}^{T}\left(  \frac{1}{L^{3}}(v\cdot\nabla
v,\phi)\right)  dt+\nonumber\\
&  +\frac{1}{T}\int_{0}^{T}\left(  \frac{1}{L^{3}}\int_{\Omega}[\nu
+\,\nu_{turb}]\nabla{v}:\nabla\phi dx\right)  dt.\label{eqFurtherStep}%
\end{align}
The main issue is the RHS model term, $\int\,\nu_{turb}\nabla{v}:\nabla\phi
dx$. Before treating that, we recall the analysis of Doering and Constantin
\cite{CD92} and Wang \cite{Wang97} for the two terms shared by the NSE,
$(v\cdot\nabla v,\phi)$\ and $\int\nu\nabla{v}:\nabla\phi dx$. For the
nonlinear term, $\frac{1}{T}\int_{0}^{T}\frac{1}{L^{3}}(v\cdot\nabla
v,\phi)dt=:NLT$, we have%
\begin{gather*}
NLT=\frac{1}{T}\int_{0}^{T}\frac{1}{L^{3}}([v-\phi]\cdot\nabla v,\phi
)dt+\frac{1}{T}\int_{0}^{T}\frac{1}{L^{3}}(\phi\cdot\nabla v,\phi)dt\\
\leq\frac{1}{T}\int_{0}^{T}\left(  \frac{1}{L^{3}}\int_{\mathcal{S}_{\beta}%
}|v-\phi||\nabla v||\phi|+|\phi|^{2}|\nabla v|dx\right)  dt\\
\leq\frac{1}{T}\int_{0}^{T}\left(  \frac{1}{L^{3}}\left\Vert \frac{v-\phi
}{L-z}\right\Vert _{L^{2}(\mathcal{S}_{\beta})}||\nabla v||_{L^{2}%
(\mathcal{S}_{\beta})}||(L-z)\phi||_{L^{\infty}(\mathcal{S}_{\beta})}\right)
dt+\\
+\frac{1}{L^{3}}\frac{1}{T}\int_{0}^{T}||\phi||_{L^{\infty}(\mathcal{S}%
_{\beta})}^{2}||\nabla v||_{L^{1}(\mathcal{S}_{\beta})}dt.
\end{gather*}
On the RHS, $||\phi||_{L^{\infty}(\mathcal{S}_{\beta})}^{2}=\phi(L)^{2}=U^{2}$
and $||(L-z)\phi||_{L^{\infty}(\mathcal{S}_{\beta})}=\frac{1}{4}\beta LU.$
Since $v-\phi$\ vanishes on the $z=L$ boundary of $\partial\mathcal{S}_{\beta
}$, Hardy's inequality, (\ref{HardyIneq}) in Section 2, the triangle
inequality and a calculation imply
\begin{align*}
\left\Vert \frac{v-\phi}{L-z}\right\Vert _{L^{2}(\mathcal{S}_{\beta})} &
\leq2\left\Vert \nabla(v-\phi)\right\Vert _{L^{2}(\mathcal{S}_{\beta})}%
\leq2\left\Vert \nabla v\right\Vert _{L^{2}(\mathcal{S}_{\beta})}+2\left\Vert
\nabla\phi\right\Vert _{L^{2}(\mathcal{S}_{\beta})}\\
&  \leq2\left\Vert \nabla v\right\Vert _{L^{2}(\mathcal{S}_{\beta})}%
+2U\sqrt{\frac{L}{\beta}}.
\end{align*}
Thus we have the estimate%
\begin{align}
NLT &  \leq\frac{\beta LU}{4}\frac{1}{L^{3}}\frac{1}{T}\int_{0}^{T}\left(
2||\nabla v||_{L^{2}(\mathcal{S}_{\beta})}^{2}+2U\sqrt{\frac{L}{\beta}%
}||v||_{L^{2}(\mathcal{S}_{\beta})}\right)  dt+\label{eq:NLTest}\\
&  +\frac{1}{L^{3}}\left(  \frac{1}{T}\int_{0}^{T}\frac{U^{2}}{L^{3}}||\nabla
v||_{L^{1}(\mathcal{S}_{\beta})}dt\right)  .\nonumber
\end{align}
For the last term on the RHS, H\"{o}lders inequality in space then in time
implies%
\begin{align*}
\frac{U^{2}}{L^{3}}\frac{1}{T}\int_{0}^{T}\int_{\mathcal{S}_{\beta}}|\nabla
v|\cdot1dxdt &  \leq\frac{U^{2}}{L^{3}}\frac{1}{T}\int_{0}^{T}\sqrt
{\int_{\mathcal{S}_{\beta}}|\nabla v|^{2}dx}\sqrt{\beta L^{3}}dt\\
&  \leq\frac{U^{2}\sqrt{\beta}}{L^{3/2}}\frac{1}{T}\int_{0}^{T}1\cdot
\sqrt{\int_{\mathcal{S}_{\beta}}|\nabla v|^{2}dx}dt\\
&  \leq\frac{U^{2}\sqrt{\beta}}{L^{3/2}}\left(  \frac{1}{T}\int_{0}^{T}%
\int_{\mathcal{S}_{\beta}}|\nabla v|^{2}dxdt\right)  ^{1/2}.
\end{align*}
Increase the integral from $\mathcal{S}_{\beta}$ to $\Omega$, and use
$\beta=\frac{1}{8}\mathcal{R}e_{eff}^{-1}.$ Rearranging and using the
arithmetic-geometric inequality gives an estimate useful for the last term in
(\ref{eq:NLTest})%
\begin{gather*}
\frac{U^{2}}{L^{3}}\frac{1}{T}\int_{0}^{T}||\nabla v||_{L^{1}(\mathcal{S}%
_{\beta})}dt\leq U^{2}\sqrt{\beta}\left(  \frac{1}{T}\int_{0}^{T}\frac
{1}{L^{3}}\int_{\Omega}|\nabla v|^{2}dxdt\right)  ^{1/2}\\
\leq U^{2}\sqrt{\frac{1}{8}\frac{1}{LU}}\left(  \frac{1}{T}\int_{0}^{T}%
\frac{1}{L^{3}}\int_{\Omega}\nu_{eff}|\nabla v|^{2}dxdt\right)  ^{1/2}\\
\leq\left(  \frac{U^{3}}{L}\right)  ^{1/2}\left(  \frac{1}{8}\frac{1}{T}%
\int_{0}^{T}\frac{1}{L^{3}}\int_{\Omega}\nu_{eff}|\nabla v|^{2}dxdt\right)
^{1/2}\\
\leq\frac{1}{2}\frac{U^{3}}{L}+\frac{1}{16}\frac{1}{T}\int_{0}^{T}\left(
\frac{1}{L^{3}}\int_{\Omega}\nu_{eff}|\nabla^{s}v|^{2}dx\right)  dt.
\end{gather*}
Similar manipulations (using $\beta=\frac{1}{8}\frac{\nu_{eff}}{LU}$) yield an
estimate for the second term on the RHS in (\ref{eq:NLTest}):%
\begin{gather*}
\frac{\beta LU}{4}\frac{1}{L^{3}}\frac{1}{T}\int_{0}^{T}\left(  2U\sqrt
{\frac{L}{\beta}}||v||_{L^{2}(\mathcal{S}_{\beta})}\right)  dt\leq\frac{\beta
LU}{2}\frac{1}{T}\int_{0}^{T}\frac{1}{L^{3}}||\nabla v||_{L^{2}(\mathcal{S}%
_{\beta})}^{2}dt+\frac{1}{8}\frac{U^{3}}{L}\\
\leq\frac{1}{8}\frac{1}{T}\int_{0}^{T}\frac{1}{L^{3}}\nu_{eff}||\nabla
v||_{L^{2}(\Omega)}^{2}dt+\frac{1}{8}\frac{U^{3}}{L}.
\end{gather*}
The first term on the RHS is simplest:%
\begin{align*}
\frac{\beta LU}{4}\frac{1}{L^{3}}\frac{1}{T}\int_{0}^{T}2||\nabla
v||_{L^{2}(\mathcal{S}_{\beta})}^{2}dt &  =\frac{1}{16}\frac{1}{L^{3}}\frac
{1}{T}\int_{0}^{T}\nu_{eff}||\nabla v||_{L^{2}(\mathcal{S}_{\beta})}^{2}dt\\
&  \leq\frac{1}{16}\frac{1}{T}\int_{0}^{T}\frac{\nu_{eff}}{L^{3}}||\nabla
v||_{L^{2}(\Omega)}^{2}dt.
\end{align*}
Using the last three estimates in the $NLT$ upper bound (\ref{eq:NLTest}), we
obtain (term by term)%
\begin{gather*}
NLT\leq\frac{1}{16}\frac{1}{T}\int_{0}^{T}\frac{\nu_{eff}}{L^{3}}||\nabla
v||_{L^{2}(\Omega)}^{2}dt+\frac{1}{8}\frac{1}{T}\int_{0}^{T}\frac{\nu_{eff}%
}{L^{3}}||\nabla v||_{L^{2}(\Omega)}^{2}dt+\\
+\frac{1}{8}\frac{U^{3}}{L}+\frac{1}{2}\frac{U^{3}}{L}+\frac{1}{16}\frac{1}%
{T}\int_{0}^{T}\frac{\nu_{eff}}{L^{3}}||\nabla v||_{L^{2}(\Omega)}^{2}dt\\
or:\text{ \ \ }NLT\leq\frac{1}{4}\frac{1}{T}\int_{0}^{T}\frac{\nu_{eff}}%
{L^{3}}||\nabla v||_{L^{2}(\Omega)}^{2}dt+\frac{5}{8}\frac{U^{3}}{L}.
\end{gather*}
Thus, from (\ref{eqFurtherStep})%
\begin{align*}
\frac{1}{T}\int_{0}^{T}\varepsilon_{\text{model}}dt &  \leq\mathcal{O}%
(\frac{1}{T})+\frac{1}{4}\frac{1}{T}\int_{0}^{T}\frac{\nu_{eff}}{L^{3}%
}||\nabla v||_{L^{2}(\Omega)}^{2}dt+\frac{5}{8}\frac{U^{3}}{L}+\\
&  +\frac{1}{T}\int_{0}^{T}\left(  \frac{1}{L^{3}}\int_{\Omega}[\nu
+\,\nu_{turb}]\nabla{v}:\nabla\phi dx\right)  dt.
\end{align*}
Consider now the last term on the RHS. Since $\phi$ is zero off $\mathcal{S}%
_{\beta}$ and $\nabla\,\phi\,=\frac{U}{\beta L}$ on $\mathcal{S}_{\beta}$\ we
have
\begin{gather*}
\frac{1}{T}\int_{0}^{T}\frac{1}{L^{3}}\int_{\Omega}[\nu+\,\nu_{turb}]\nabla
{v}:\nabla\phi dxdt=\frac{1}{T}\int_{0}^{T}\frac{1}{L^{3}}\int_{\mathcal{S}%
_{\beta}}[\nu+\,\nu_{turb}]\nabla{v}:\nabla\phi dxdt\\
\leq\frac{1}{2}\frac{1}{T}\int_{0}^{T}\varepsilon_{\text{model}}dt+\frac{1}%
{2}\frac{1}{T}\int_{0}^{T}\left(  \frac{1}{L^{3}}\int_{\mathcal{S}_{\beta}%
}[\nu+\,\nu_{turb}]\left(  \frac{U}{\beta L}\right)  ^{2}dx\right)  dt\\
\leq\frac{1}{2}\frac{1}{T}\int_{0}^{T}\varepsilon_{\text{model}}dt+\frac{1}%
{2}\left(  \frac{U}{\beta L}\right)  ^{2}\beta\frac{1}{T}\int_{0}^{T}\left(
\frac{1}{\beta L^{3}}\int_{\mathcal{S}_{\beta}}\nu+\,\nu_{turb}dx\right)  dt.
\end{gather*}
Thus, as $\beta=\frac{1}{8}\mathcal{R}e_{eff}^{-1}$\ implies $2\beta
\mathcal{R}e_{eff}=1/4$,
\begin{gather*}
\frac{1}{2}\frac{1}{T}\int_{0}^{T}\varepsilon_{\text{model}}dt\leq
\mathcal{O}(\frac{1}{T})+\frac{1}{4}\frac{1}{T}\int_{0}^{T}\frac{\nu_{eff}%
}{L^{3}}||\nabla v||^{2}dt+\\
+\frac{5}{8}\frac{U^{3}}{L}+\frac{\beta}{2}\left(  \frac{U}{\beta L}\right)
^{2}\frac{1}{T}\int_{0}^{T}\frac{1}{\beta L^{3}}\int_{\mathcal{S}_{\beta}}%
\nu+\,\nu_{turb}dxdt.
\end{gather*}
As a subsequence $T_{j}\rightarrow\infty$%
\[
\frac{1}{T}\int_{0}^{T}\frac{\nu_{eff}}{L^{3}}||\nabla v||^{2}dt\rightarrow
\left\langle \varepsilon_{\text{model}}\right\rangle _{\infty},
\]
and we calculate%
\[
\frac{\beta}{2}\left(  \frac{U}{\beta L}\right)  ^{2}\frac{1}{T}\int_{0}%
^{T}\frac{1}{\beta L^{3}}\int_{\mathcal{S}_{\beta}}\nu dxdt=\nu\frac{\beta}%
{2}\left(  \frac{U}{\beta L}\right)  ^{2}=4\frac{\nu}{\nu_{eff}}\frac{U^{3}%
}{L}.
\]
Thus,%
\begin{align*}
\frac{1}{2}\frac{1}{T}\int_{0}^{T}\varepsilon_{\text{model}}dt &
\leq\mathcal{O}(\frac{1}{T})+\frac{1}{4}\frac{1}{T}\int_{0}^{T}\frac{\nu
_{eff}}{L^{3}}||\nabla v||_{L^{2}(\Omega)}^{2}dt\\
&  +(4\frac{\nu}{\nu_{eff}}+\frac{5}{8})\frac{U^{3}}{L}+\beta\left(  \frac
{U}{\beta L}\right)  ^{2}\frac{1}{T}\int_{0}^{T}\frac{1}{\beta L^{3}}%
\int_{\mathcal{S}_{\beta}}\,\nu_{turb}dxdt.
\end{align*}
The last term on the RHS is rearranged to be%
\[
\beta\left(  \frac{U}{\beta L}\right)  ^{2}\frac{1}{T}\int_{0}^{T}\frac
{1}{\beta L^{3}}\int_{\mathcal{S}_{\beta}}\,\nu_{turb}dxdt=8\left[  \frac
{1}{T}\int_{0}^{T}\frac{1}{|\mathcal{S}_{\beta}|}\int_{\mathcal{S}_{\beta}%
}\,\frac{\nu_{turb}}{\nu_{eff}}dxdt\right]  \frac{U^{3}}{L},
\]
The proof is completed by taking the limit superior as $T\rightarrow\infty$.
This gives%
\[
\left\langle \varepsilon_{\text{model}}\right\rangle _{\infty}\leq\left(
\frac{5}{2}+16\frac{\nu}{\nu_{eff}}+32\left\langle \frac{1}{|\mathcal{S}%
_{\beta}|}\int_{\mathcal{S}_{\beta}}\,\frac{\nu_{turb}}{\nu_{eff}%
}dx\right\rangle _{\infty}\right)  \frac{U^{3}}{L}.
\]

\end{proof}

The above theorem shows that the behavior of the turbulent viscosity in the
near wall region is a critical factor in model dissipation. The near wall term
will be estimated by a precise use of the following lemma which is an
application of the $L^{p}-L^{q}$ Hardy inequality (\ref{eq:HardyLpLq}).

\begin{lemma}
For $F(x,y,z)\in H^{1}(\mathcal{S}_{\beta})$ with $F(x,y,L)=0$ there holds%
\begin{equation}
\left(  \int_{(1-\beta)L}^{L}|L-z|^{-4}|F(x,y,z)|^{6}dz\right)  ^{1/6}\leq
C_{26}\left(  \int_{(1-\beta)L}^{L}|F_{z}(x,y,z)|^{2}dz\right)  ^{1/2}%
\label{eq:NearWallIneq}%
\end{equation}

\end{lemma}

\begin{proof}
First note that for $F(x)$ vanishing at $x=0$ there holds%
\[
\left(  \int_{0}^{\infty}\left\vert \frac{F(x)}{x^{2/3}}\right\vert
^{6}dx\right)  ^{1/6}\leq C_{26}\left(  \int_{0}^{\infty}\left\vert
F_{x}(x)\right\vert ^{2}dx\right)  ^{1/2}.
\]
This follows from (\ref{eq:HardyLpLq}) by letting $F(x)=\int_{0}^{x}f(t)dt$
and choosing $q=6,p=2,$ $\alpha=$ $-4$ . These choices satisfy $\left(
\alpha+1\right)  /6=(1/2)-1$ and $1<p\leq q<\infty$ as required. We will apply
this result where $F$ is a function of $x,y,z$ defined on $\mathcal{S}_{\beta
}$, the variable $x$ is replaced by $z$ and a linear change of variables is
made so the zero boundary condition is imposed at $z=L$ rather than $z=0$. The
result is%
\[
\left(  \int_{(1-\beta)L}^{L}|L-z|^{-4}|F(x,y,z)|^{6}dz\right)  ^{1/6}\leq
C_{26}\left(  \int_{(1-\beta)L}^{L}|F_{z}(x,y,z)|^{2}dz\right)  ^{1/2}%
\]

\end{proof}

\bigskip This lemma will be used to estimate the integral
\[
\int_{\mathcal{S}_{\beta}}\,\nu_{turb}dx=\mu\tau\int_{\mathcal{S}_{\beta}%
}\,|u^{\prime}(x,t)|_{e}^{2}dx.
\]

\begin{lemma}
We have
\begin{align*}
\int_{\mathcal{S}_{\beta}}\,|u^{\prime}(x,y,z,t;\omega_{j})|^{2}dx  &
\leq3^{-2/3}C_{26}^{2}\beta^{2}L^{2}\int_{\mathcal{S}_{\beta}}\left\vert
\frac{\partial u^{\prime}}{\partial z}(x,y,z,t;\omega_{j})\right\vert
^{2}dx,\\
\int_{\mathcal{S}_{\beta}}\,|u^{\prime}(x,y,z,t)|_{e}^{2}dx  &  \leq
3^{-2/3}C_{26}^{2}\beta^{2}L^{2}\int_{\mathcal{S}_{\beta}}\left\langle
\,\left\vert \frac{\partial u^{\prime}}{\partial z}(x,y,z,t;\omega
_{j})\right\vert ^{2}\right\rangle _{e}dx
\end{align*}
and thus%
\[
\int_{\mathcal{S}_{\beta}}\,\nu_{turb}dx\leq\mu\tau3^{-2/3}C_{26}^{2}\beta
^{2}L^{2}\int_{\mathcal{S}_{\beta}}\left\langle \,\left\vert \frac{\partial
u^{\prime}}{\partial z}(x,y,z,t;\omega_{j})\right\vert ^{2}\right\rangle
_{e}dx
\]

\end{lemma}

\begin{proof}
We have for each $\omega_{j}$, by Holders inequality in the $z$ integral,%
\begin{gather*}
\int_{\mathcal{S}_{\beta}}\,|u^{\prime}|^{2}dx=\int_{0}^{L}\int_{0}^{L}\left[
\int_{(1-\beta)L}^{L}\left(  L-z\right)  ^{4/3}\,\left\vert \frac{u^{\prime
}(x,y,z,t;\omega_{j})}{\left(  L-z\right)  ^{2/3}}\right\vert ^{2}dz\right]
dxdy\leq\\
\int_{0}^{L}\int_{0}^{L}\left[  \int_{(1-\beta)L}^{L}\left(  \left(
L-z\right)  ^{4/3}\right)  ^{3/2}\,dz\right]  ^{\frac{2}{3}}\left[
\int_{(1-\beta)L}^{L}\,\left\vert \frac{u^{\prime}(x,y,z,t;\omega_{j}%
)}{\left(  L-z\right)  ^{2/3}}\right\vert ^{6}dz\right]  ^{\frac{1}{3}}dxdy\\
:=\int\int A\cdot Bdxdy.
\end{gather*}
The first integral ($A$) in the RHS is%
\[
A=\left[  \int_{(1-\beta)L}^{L}\left(  L-z\right)  ^{2}\,dz\right]
^{2/3}=\left[  \frac{\left(  \beta L\right)  ^{3}}{3}\right]  ^{2/3}%
=3^{-2/3}\beta^{2}L^{2}.
\]
For the second integral ($B$) we apply (\ref{eq:NearWallIneq}) from the last
Lemma:%
\begin{align*}
B  &  =\left[  \int_{(1-\beta)L}^{L}\left(  L-z\right)  ^{-4}\left\vert
u^{\prime}(x,y,z,t;\omega_{j})\right\vert ^{6}dz\right]  ^{1/3}\\
&  \leq C_{26}^{2}\int_{(1-\beta)L}^{L}\,\left\vert \frac{\partial u^{\prime}%
}{\partial z}(x,y,z,t;\omega_{j})\right\vert ^{2}dz,
\end{align*}
In combination we have%
\[
\int_{\mathcal{S}_{\beta}}\,|u^{\prime}(x,t;\omega_{j})|^{2}dx\leq\left(
3^{-2/3}\beta^{2}L^{2}\right)  C_{26}^{2}\int_{0}^{L}\int_{0}^{L}%
\int_{(1-\beta)L}^{L}\,\left\vert \frac{\partial u^{\prime}}{\partial
z}(x,t;\omega_{j})\right\vert ^{2}dzdxdy
\]
which is the first estimate. The second follows by taking the ensemble average
of the first.
\end{proof}

We can now estimate the last term on the RHS of (\ref{eq:ShearEstimate}).

\begin{proposition}
We have%
\begin{gather*}
32\left\langle \frac{1}{|\mathcal{S}_{\beta}|}\int_{\mathcal{S}_{\beta}%
}\,\frac{\nu_{turb}}{\nu_{eff}}dx\right\rangle _{\infty}\frac{U^{3}}{L}\leq\\
\leq\frac{\sqrt[3]{3}}{6}C_{26}^{2}\mu\frac{\tau}{T^{\ast}}\left\langle
\frac{1}{|\mathcal{S}_{\beta}|}\int_{\mathcal{S}_{\beta}}\nu_{eff}\left\langle
\,\left\vert \frac{\partial u^{\prime}}{\partial z}\right\vert ^{2}%
\right\rangle _{e}dx\right\rangle _{\infty}\\
\leq\frac{\sqrt[3]{3}}{6}C_{26}^{2}\mu\frac{\tau}{T^{\ast}}\left\langle
\frac{1}{|\mathcal{S}_{\beta}|}\int_{\mathcal{S}_{\beta}}\nu_{eff}\left\langle
\,\left\vert \nabla u\right\vert ^{2}\right\rangle _{e}dx\right\rangle
_{\infty}.
\end{gather*}

\end{proposition}

\begin{proof}
From the above estimates we have%
\begin{gather*}
32\left\langle \frac{1}{|\mathcal{S}_{\beta}|}\int_{\mathcal{S}_{\beta}%
}\,\frac{\nu_{turb}}{\nu_{eff}}dx\right\rangle _{\infty}\frac{U^{3}}{L}\leq\\
\leq\left(  32\cdot3^{-2/3}C_{26}^{2}\right)  \beta^{2}L^{2}\frac{\mu\tau}%
{\nu_{eff}}\frac{1}{|\mathcal{S}_{\beta}|}\int_{\mathcal{S}_{\beta}%
}\left\langle \,\left\vert \frac{\partial u^{\prime}}{\partial z}\right\vert
^{2}\right\rangle _{e}dx\frac{U^{3}}{L}.
\end{gather*}
Algebraic rearrangement of various multipliers in the above (using turnover
time $T^{\ast}=L/U$) gives%
\[
\beta^{2}L^{2}\frac{\mu\tau}{\nu_{eff}}\frac{U^{3}}{L}=\left(  \frac{\nu
_{eff}}{8LU}\right)  ^{2}L^{2}\frac{1}{\nu_{eff}}\frac{\mu\tau}{T^{\ast}}%
\frac{L}{U}\frac{U^{3}}{L}=\frac{1}{64}\nu_{eff}\frac{\mu\tau}{T^{\ast}}.
\]
Thus%
\begin{gather*}
32\left\langle \frac{1}{|\mathcal{S}_{\beta}|}\int_{\mathcal{S}_{\beta}%
}\,\frac{\nu_{turb}}{\nu_{eff}}dx\right\rangle _{\infty}\frac{U^{3}}{L}\leq\\
\leq\left(  \frac{32}{64}\cdot3^{-2/3}C_{26}^{2}\right)  \mu\frac{\tau
}{T^{\ast}}\frac{1}{|\mathcal{S}_{\beta}|}\int_{\mathcal{S}_{\beta}}\nu
_{eff}\left\langle \,\left\vert \frac{\partial u^{\prime}}{\partial
z}\right\vert ^{2}\right\rangle _{e}dx.
\end{gather*}
The result now follows from%
\[
\left\langle \,\left\vert \frac{\partial u^{\prime}}{\partial z}\right\vert
^{2}\right\rangle _{e}\leq\left\langle \,\left\vert \frac{\partial u}{\partial
z}\right\vert ^{2}\right\rangle _{e}\leq\left\langle \,\left\vert \nabla
u\right\vert ^{2}\right\rangle _{e}.
\]

\end{proof}

To develop the numerical value for $\frac{\sqrt[3]{3}}{6}C_{26}^{2}\simeq$
$\allowbreak0.230\,92$\ we begin with $\frac{\sqrt[3]{3}}{6}\simeq
\allowbreak0.240\,37$ and, from Bliss \cite{B30}, with $p=p^{\prime}=2,q=6$
\begin{equation}
C_{pq}=\left(  \frac{p^{\prime}}{q}\right)  ^{1/p}\left(  \frac{\frac{q-p}%
{p}\Gamma\left(  \frac{pq}{q-p}\right)  }{\Gamma\left(  \frac{p}{q-p}\right)
\Gamma\left(  \frac{p(q-1)}{q-p}\right)  }\right)  \label{eq:OptimalCpq}%
\end{equation}
so that%
\[
C_{26}=\frac{1}{\sqrt{3}}\left(  \frac{2\Gamma\left(  3\right)  }%
{\Gamma\left(  \frac{1}{2}\right)  \Gamma\left(  \frac{5}{2}\right)  }\right)
\simeq0.980\,14.
\]

We can now give another result for the time averaged energy dissipation rate.

\begin{theorem}
Any weak solution of the eddy viscosity model (\ref{eq:EnsembleNSE})
satisfying the energy inequality (\ref{eq:ShearEnergyIneq}) has its model
energy dissipation bounded as%
\begin{gather}
\left\langle \varepsilon_{\text{model}}\right\rangle _{\infty}\leq\left(
\frac{5}{2}+16\frac{\nu}{\nu_{eff}}\right)  \frac{U^{3}}{L}+\nonumber\\
+\left(  \frac{\sqrt[3]{3}}{6}C_{26}^{2}\right)  \mu\frac{\tau}{T^{\ast}%
}\left\langle \frac{1}{|\mathcal{S}_{\beta}|}\int_{\mathcal{S}_{\beta}}%
\nu_{eff}\left\langle \,\left\vert \nabla u\right\vert ^{2}\right\rangle
_{e}dx\right\rangle _{\infty}\label{eq:ThirdStep}\\
and\nonumber\\
\left\langle \varepsilon_{\text{model}}\right\rangle _{\infty}\leq\left(
\frac{5}{2}+16\frac{\nu}{\nu_{eff}}\right)  \frac{U^{3}}{L}+\nonumber\\
+\left(  \frac{\sqrt[3]{3}}{6}C_{26}^{2}\right)  \mu\frac{\tau}{T^{\ast}%
}\left\langle \frac{1}{|\mathcal{S}_{\beta}|}\int_{\mathcal{S}_{\beta}}%
\nu_{eff}\left\langle \,\left\vert \frac{\partial u^{\prime}}{\partial
z}\right\vert ^{2}\right\rangle _{e}dx\right\rangle _{\infty}%
\end{gather}

\end{theorem}

\begin{proof}
This follows by combining the previous estimates.
\end{proof}

To obtain a closed estimate we now consider choosing a different $\mu-$value,
$\mu=\mu_{\beta}$, in the near wall region than away from walls. This becomes
necessary because the integral over $\mathcal{S}_{\beta}$\ must be increased
to one over $\Omega$.

\begin{theorem}
Suppose $\mu$ is one constant value in the flow interior and $\mu=\mu_{\beta}$
in $\mathcal{S}_{\beta}$. If $\mu_{\beta}\leq0.270\,64\mathcal{R}e^{-1}$ then
\[
\left\langle \varepsilon_{\text{model}}\right\rangle _{\infty}\leq\left(
5+32\frac{\nu}{\nu_{eff}}\right)  \frac{U^{3}}{L}.
\]
$.$
\end{theorem}

\begin{proof}
To begin note that if $\mu$\ is piecewise constant, the previous estimates
hold with $\mu(x)$\ inside the volume integral or the volume integral split
into sub-regions $\mathcal{S}_{\beta}$\ and $\Omega-\mathcal{S}_{\beta}$. We
estimate the last term on the RHS of (\ref{eq:ThirdStep}) (increasing the
integral to over $\Omega$) as follows%
\begin{gather*}
\left(  \frac{\sqrt[3]{3}}{6}C_{26}^{2}\right)  \mu_{\beta}\frac{\tau}%
{T^{\ast}}\left\langle \frac{1}{|\mathcal{S}_{\beta}|}\int_{\mathcal{S}%
_{\beta}}\nu_{eff}\left\langle \,\left\vert \nabla u\right\vert ^{2}%
\right\rangle _{e}dx\right\rangle _{\infty}\leq\\
\left(  \frac{8\sqrt[3]{3}}{6}C_{26}^{2}\right)  \mu_{\beta}\frac{\tau
}{T^{\ast}}\mathcal{R}e_{eff}\left\langle \frac{1}{L^{3}}\int_{\mathcal{\Omega
}}\nu_{eff}\left\langle \,\left\vert \nabla u\right\vert ^{2}\right\rangle
_{e}dx\right\rangle _{\infty}\\
\leq\left(  \frac{8\sqrt[3]{3}}{6}C_{26}^{2}\right)  \mu_{\beta}\frac{\tau
}{T^{\ast}}\mathcal{R}e_{eff}\left\langle \varepsilon_{\text{model}%
}\right\rangle _{\infty}=\left(  \frac{8\sqrt[3]{3}}{6}C_{26}^{2}\right)
\mu_{\beta}\frac{\tau}{T^{\ast}}\mathcal{R}e\left(  \frac{\nu}{\nu_{eff}%
}\right)  \left\langle \varepsilon_{\text{model}}\right\rangle _{\infty}%
\end{gather*}
The multiplier (to 5 digits) $\frac{8C_{26}^{2}\sqrt[3]{3}}{6}\simeq
1.\,\allowbreak847\,4$. Note that
\[
\frac{\nu}{\nu_{eff}}\leq1\text{ and }\frac{\tau}{T^{\ast}}\leq1.
\]
Thus, if $\mu_{\beta}$\ is chosen so that
\[
\frac{8C_{26}^{2}\sqrt[3]{3}}{6}\mu_{\beta}\mathcal{R}e\leq\frac{1}{2}\text{
\ implied by }\mu_{\beta}\leq0.270\,64\mathcal{R}e^{-1}%
\]
it follows that%
\[
\left\langle \varepsilon_{\text{model}}\right\rangle _{\infty}\leq\left(
5+32\frac{\nu}{\nu_{eff}}\right)  \frac{U^{3}}{L}\text{ \ }\left(  \leq
37\frac{U^{3}}{L}\right)  .
\]

\end{proof}

\section{Conclusions}

We have analyzed wall effects on energy dissipation rates for the ensemble
eddy viscosity model. The results herein are only a first step. There are a
number of important mathematical challenges (open problems) that remain:

\begin{itemize}
\item The constant multipliers in the results are large as a result of the
number of estimates employed. Significant reduction would bring the analytical
estimates closer to experimental data.

\item The model is based on a turbulence length scale $l(x,t)=|u^{\prime
}(x,t)|_{e}\tau$. Since this is an $L^{2}$ function, there is no mathematical
reason preventing this length scale predicting (locally) eddies farther apart
than the domain size! In \cite{LR26} to prove existence of a model solution it
was found necessary to cap $l$ at the domain size by $l(x,t)=\min\{|u^{\prime
}(x,t)|_{e}\tau,L\}.$ This can even be reasonably modified to $l(x,t)=\min
\{|u^{\prime}(x,t)|_{e}\tau,d(x)\}$. Analytical evaluation of the effects of
such caps on energy dissipation rates is another important open problem.

\item We conjecture that the assumption $\mu_{\beta}\lesssim\mathcal{R}e^{-1}%
$\ is not necessary. However, the work herein suggests that dropping
$\mu_{\beta}\lesssim\mathcal{R}e^{-1}$ will require a new idea at some point.

\item Based on the results in \cite{CRW02} (that $J=16$ was large enough to
capture low order flow statistics) we have considered $J$ fixed. The analysis
of the limit $J\rightarrow\infty$\ is an important open problem.

\item For shear flow, existence of weak solutions and their energy inequality
are open problems in analysis.

\item The model herein, like all EV models, is dissipative and thus cannot
account for intermittent transfer of energy from unresolved fluctuations back
to the mean velocity. Attempts have been made based on negative viscosities
but these cannot be correct in principle. A more correct approach is through
an exact equation for variance evolution, developed in \cite{JL16}. Analysis
of these model extensions to account for intermittence is an open problem..
\end{itemize}

\end{document}